\documentclass[11pt]{amsproc}
\usepackage{amsfonts}
\usepackage{amsmath,amstext,amsbsy,amssymb}

\newtheorem{theorem}{Theorem}[section]
\newtheorem{lemma}[theorem]{Lemma}

\newtheorem{proposition}[theorem]{Proposition}

\theoremstyle{definition}
\newtheorem{definition}[theorem]{Definition}
\newtheorem{example}[theorem]{Example}

 \newtheorem{conjecture}[theorem]{Conjecture}

\theoremstyle{remark}
\newtheorem{remark}[theorem]{Remark}

\numberwithin{equation}{section}

\newcommand{\be}{\begin{enumerate}}
\newcommand{\ee}{\end{enumerate}}
\newcommand{\la}{\lambda}
\newcommand{\al}{\alpha}
\newcommand{\zz}{\mathbb{Z}}
\newcommand{\qq}{\mathbb{Q}}
\begin{document}

\title[The Smith Normal Form of a Matrix]{The Smith Normal Form of a
Matrix Associated with Young's Lattice}

\author{Tommy Wuxing Cai}
\address{School of Sciences,
SCUT, Guangzhou 510640, China}
\email{caiwx@scut.edu.cn}

\author{Richard P. Stanley}

\thanks{The first author thanks M.I.T.\ for providing a great research
environment, the China Scholarship Council for partial support, and
the Combinatorics Center at Nankai University for their hospitality
when this work was initiated.  The second author was partially
supported by NSF grant DMS-1068625.}

\address{Department of mathematics,
M.I.T., Cambridge, MA, USA}
\email{rstan@math.mit.edu}

\begin{abstract}
We prove a conjecture of Miller and Reiner on the Smith normal form of
the operator $DU$ associated with a differential poset for the special
case of Young's lattice. Equivalently, this operator can be
described as $\frac{\partial}{\partial p_1}p_1$ acting on homogeneous
symmetric functions of degree $n$.
\end{abstract}

\maketitle

\section{Introduction}

Let $R$ be a commutative ring with 1 and $M$ an $m\times m$ matrix
over $R$. We say that $M$ has a \emph{Smith normal form} (SNF) over
$R$ if there exist matrices $P,Q\in \mathrm{SL}(m,R)$ (so $\det P=\det
Q=1$) such that $PMQ$ is a diagonal matrix
$\mathrm{diag}(d_1,d_2,\dotsc,d_m)$ with $d_i\mid d_{i+1}$ ($1\leq
i\leq m-1$) in $R$. It is well-known that if $R$ is an integral domain
and if the SNF of $M$ exists, then it is unique up to multiplication
of each $d_i$ by a unit $u_i$ (with $u_1\cdots u_m=1$).

We see that Smith normal form is a refinement of the determinant,
since $\det M=d_1d_2\cdots d_m$.  In the case that $R$ is a principal
ideal domain (PID), it is well-known that all matrices in $R$ have a
Smith normal form. Not very much is known in general.

In this work we are interested in the ring $\zz[x]$ of
integer polynomials in one variable. Since $\zz[x]$ is not a PID, a
matrix over this ring need not have an SNF. For instance,
it can be shown that the
matrix $\left[\begin{array}{cc} x & 0  \\ 0 & x+2\\
\end{array}\right]$ does not have a Smith normal form over $\zz[x]$.

We will use symmetric function notation and terminology
from \cite{ec2}.  The ring $\Lambda$ of symmetric functions has
several standard $\zz$-bases: monomial symmetric functions $m_\la$,
elementary symmetric functions $e_\la$, complete symmetric functions
$h_\la$ and Schur functions $s_\la$. The power sum symmetric functions
$p_\la$ form a $\qq$-basis of $\Lambda_{\qq}=\Lambda\otimes_\zz \qq$.
The ring $\Lambda_{\qq}$ is a graded $\qq$-algebra:
$\Lambda_{\qq}=\bigoplus_{n=0}^\infty \Lambda_{\qq}^n$, where
$\Lambda_{\qq}^n$ is the vector space spanned over $\qq$ by
$\{s_\la:\la\vdash n\}$.  Similarly, $\Lambda$ is a graded ring.

Regard elements of $\Lambda_\qq^n$ as polynomials in the $p_j$'s.
Define a linear map $T_1$ on $\Lambda_{\qq}^n$ by
\begin{align}\label{D:p1*p1}
T_1(v)=\frac{\partial}{\partial p_1}(p_1v),\ \ v\in\Lambda_{\qq}^n.
\end{align}

\begin{remark} \label{psum}
Note that the power sum $p_\lambda$ is an eigenvector
of $T_1(v)$ with corresponding eigenvalue $1+m_1(\lambda)$, where
$m_1(\lambda)$ denotes the number of 1's in $\lambda$.
\end{remark}

Denote by $M=M^{(n)}_1$ the matrix of $T_1$ with respect to the basis
$\{s_\la:\la\vdash n\}$ (arranged in, say, lexicographic order).
It is known and easy to see that $M$ is an integer symmetric matrix of
order $p(n)$,
the number of partitions of $n$. Let
$\la(n)=(p(n)-p(n-1),\dotsc,p(2)-p(1),p(1))$, so $\lambda(n)$ is a
partition of $p(n)$. Let the conjugate of $\la(n)$ be
$\la(n)'=(j_{p(n)-p(n-1)},\dotsc,j_2,j_1)$ (so
$j_{p(n)-p(n-1)}=n$). We prove the following  result.
\begin{theorem} \label{thm1}
Let $\al_k(x)=a_1(x)a_2(x)\dotsm a_k(x)$ with $a_i(x)=i+x$
($i=1,2,\dotsc,n-1$) and $a_n(x)=n+1+x$.
There exist $P(x),Q(x)\in \mathrm{SL}(p(n),\zz[x])$ such that
$P(x)(M+xI_{p(n)})Q(x)$ is the following diagonal matrix:
\begin{align}
\mathrm{diag}(1,1,\dotsc,1,\al_{j_1}(x),\al_{j_2}(x),\dotsc,\al_{j_{p(n)-p(n-1)}}(x)).
\end{align}

\end{theorem}
As an example of Theorem~\ref{thm1}, let us consider the case $n=6$. 
First $\la(6)=(4,2,2,1,1,1)$ and $\la(6)'=(6,3,1,1)$.
Hence the diagonal entries of an SNF of $M+xI$ are seven 1's
followed by
 $$1+x,1+x,(1+x)(2+x)(3+x),(1+x)(2+x)(3+x)(4+x)(5+x)(7+x).$$

In general, the number of diagonal entries equal to 1 is
$p(n-1)$.







The origin of Theorem~\ref{thm1} is as follows. Let $P$ be a
differential poset, as defined in \cite{dposet}
or \cite[{\S}3.21]{ec1}, with levels $P_0, P_1,\dots$. Let $\zz[x]P_n$
denote the free $\zz[x]$-module with basis $P_n$. Let $U,D$ be the up
and down operators associated with $P$. Miller and Reiner \cite{MR}
conjectured a certain Smith normal form of the operator
$DU+xI\colon \zz[x] P_n\to\zz[x] P_n$. Our result is equivalent to the
conjecture of Miller and Reiner for the special case of Young's
lattice $Y$. Our result also specializes to a proof of Miller's
Conjecture 14 in \cite{M}.  After proving the theorem, we state a
conjecture which generalizes it. It seems natural to try to generalize
our work to the differential poset $Y^r$ for $r\geq 2$, but we have
been unable to do this.

\section{The proof of the theorem}
Instead of Schur functions, we consider the matrix with respect to the
complete symmetric functions $\{h_\la\colon\lambda\vdash n\}$.  Since
$\{ s_\lambda\colon\lambda\vdash n\}$ and
$\{h_\lambda\colon\lambda\vdash n\}$ are both $\zz$-bases for
$\Lambda^n$, the Smith normal form does not change when we switch to
the $h_\lambda$ basis. We introduce a new ordering on the set
$\mathcal{P}_n$ of all partitions of $n$. The matrix $A$ with
respect to this new ordering turns out to be much easier to manipulate
than the original matrix. In fact, we show that $A+xI_{p(n)}$ can be turned into an upper
triangular matrix after some simple row operations. Then we use
more row/column operations to cancel the non-diagonal elements.
The resulting diagonal matrix is the SNF that we are looking for.

From now on, we fix a positive integer $n$. The case $n=1$ is trivial,
so we assume $n\geq 2$.

\subsection{A new ordering on partitions}
\begin{definition}
Let $\la=(\la_1,\la_2\dots,\la_i,1)$ with
 $\la_1\geq\la_2\geq\dotsm\geq\la_i\geq1$ ( $i\geq 1$).
We define $\la^+=(\la_1+1,\la_2,\dots,\la_i)$ and write
 $\la^+\nwarrow\la$.
We call a partition $\la$ \emph{initial} if there is no $\mu$ such that
 $\mu^+=\la$, i.e., $\lambda_1=\lambda_2$. We call a partition
 $\mu$ \emph{terminal} if $\mu^+$ is not
 well-defined, i.e., $m_1(\mu)=0$, where $m_i(\lambda)$ denotes the
 number of parts of $\lambda$ equal to $i$.
\end{definition}

For a sequence $\la^0, \la^1,\la^2,\dotsc,\la^t$ of partitions, we
write
\begin{align}\label{E:RS}
\la^t\nwarrow\la^{t-1}\nwarrow\cdots\nwarrow\la^0
 \end{align}
 and call it a \emph{rising string} of length $t$ if
 $\la^{i+1}\nwarrow \la^{i}$, i.e., $\la^{i+1}=(\la^i)^+$, for
 $0\leq i\leq t-1$.

In equation (\ref{E:RS}), if $\la^0$ is  initial and $\la^t$ is
terminal then we call \eqref{E:RS} a \emph{full} (rising)
string. Moreover, we say that $\la^0$ is the initial element of the
string and $\la^t$ is the terminal element of the string.
One cannot add a partition to a full string to make it longer.

A partition $\mu=(\mu_1,\mu_2,\dots,\mu_i)$ with
$\mu_1=\mu_2\geq\cdots\geq\mu_i\geq2$ is not in any rising string of
nonzero length. We also call this $\mu$ a (point) string of length
0; it is both initial and terminal.

Note that every partition $\mu$ is in exactly one full string; we
denote the terminal element of this string by $T(\mu)$. Thus
$T(\la)=T(\mu)$ if and only if $\la,\mu$ are in the same string.

\begin{definition}
We define a total ordering $\preceq$ on the set of partitions of $n$
as follows.
 \be
 \item For two partitions in the rising string (\ref{E:RS}), we
 define $\la^i\preceq\la^j$ for  $i\leq j$.
 \item For $\la,\mu$ in different full strings, we write
 $\la\preceq\mu$ if $T(\la)\leq_L T(\mu)$, i.e., $T(\la)$ is
 lexicographically less than $T(\mu)$.
 \ee
\end{definition}

We arrange the partitions in $\mathcal{P}_n$ from largest to smallest
according to the following order:
\begin{align}\label{F:neworder}
\la^{11},\la^{12},\dotsc,\la^{1i_1}; \la^{21},\la^{22},\dotsc,\la^{2i_2};\dotsc;\la^{t1},\la^{t2},\dotsc,\la^{ti_t},
\end{align}
where $\la^{j1}\nwarrow\la^{j2}\nwarrow\cdots\nwarrow\la^{ji_j}$ is
the $j$th full string of partitions of $n$, and we use semicolons to
separate neighboring strings.  It's easy to see that $\la^{11}=(n)$,
$i_1=n$, $\la^{1i_1}=(1^n)$, and $t$ is the number of terminal
elements of $n$, which is equal to the number of partitions of $n$
with no part equal to $1$, viz., $p(n)-p(n-1)$. In fact, these
cardinalities $i_1,i_2,\dotsc, i_t$ of strings can be expressed
explicitly.  We will see this point after the following example.

\begin{example} \label{E:P6-ordering}
The following is the list of partitions of $6$:
 $$ 6,\ 51,\ 41^2,\ 31^3,\ 21^4,\ 1^6;\ 42,\ 321,\ 2^21^2;\ 3^2;\ 2^3.$$
The eigenvalues of $M$ arranged in accordance with this order are by
Remark~\ref{psum} as follows:
$$1,2,3,4,5,7;1,2,3;1;1.$$

On the other hand, we can arrange the eigenvalues of $M$ in the
following form:
\begin{align*}
p(6)-p(5)=4:&\;1\;\;1\;\;1\;\;1\\
p(5)-p(4)=2:&\;2\;\;2\\
p(4)-p(3)=2:&\;3\;\;3\\
p(3)-p(2)=1:&\;4\\
p(2)-p(1)=1:&\;5\\
p(1)=1:&\;7.
\end{align*}
We see that the eigenvalues (to the right side of the colons) form a
tableau (with constant rows and increasing columns) of shape
$\la(6)=(4,2,2,1,1,1)$.

Notice that the eigenvalues associated with the first string are
1,2,3,4,5, and 7 and they form the first column of the above
tableau. In fact, the eigenvalues corresponding to every string form a
column of the tableau. This is not a coincidence, but rather because
the eigenvalues corresponding to a string form a sequence of consecutive
integers starting from $1$: $1,2,3\dotsc,i$, except there is a gap for
those corresponding to the first string. Therefore the cardinality of
a string is the length of a column of the tableau.
\end{example}
We can easily formalize this argument to prove the following.
\begin{lemma}\label{L:stringcardinality}
Rearrange the cardinalities  of the strings $i_1,i_2,\dotsc, i_t$ in
(\ref{F:neworder}) in weakly decreasing order:
$J=(j_t,\dotsc,j_2,j_1)$. Then $J$ is exactly the conjugate of the
partition $\la(n)=(p(n)-p(n-1),\dotsc, p(2)-p(1),p(1))$ as defined in
the introduction.
\end{lemma}
Note that $i_1,i_2,\dotsc,i_t$ are not necessarily in weakly
decreasing order.

\subsection{The transition matrix with respect to the new ordering}

Let $A=(a_{\mu\la})$ be the $p(n)\times p(n)$ matrix of the action of
 $\frac{\partial}{\partial p_1}p_1$ on the basis $h_\la$, i.e.,
 $\frac{\partial}{\partial p_1}p_1\cdot h_\la=\sum_\mu a_{\mu\la}h_\mu$. Here we use the total
 ordering $\preceq$, with the greatest partition $(n)$ corresponding
 to the first row and column. Recall that the notation
 $\frac{\partial}{\partial p_1}v$ means that we write $v$ as a
 polynomial in the power sums $p_1,p_2,\dots$ and regard each $p_i$,
 $i\geq 2$, as a constant when we differentiate.

 \begin{lemma}\label{L:PropertyOfA}
 The matrix $A=(a_{\mu\la})$ has the following properties.
  \be
 \item $a_{\mu\la}\neq 0$ only if $\la\nwarrow\mu$, $\mu=\la$ or
 $T(\mu)>T(\la)$ (greater but not equal in dominance order).
 \item $a_{\mu\la}=1$ if $\la\nwarrow\mu$;\\
     $a_{\la\la}=m_1(\la)+1$;\\ 
     $m_1(\mu)$ equals $m_1(\la)+1$ or $m_1(\la)+2$ if $a_{\mu\la}\neq
 0$ and $T(\mu)>T(\la)$.
 \ee
 \end{lemma}
\begin{proof}
Let $h_\la=h_{\la_1}\dotsm h_{\la_k}\dotsm h_{\la_i}h_1^j$, with
$\la_1\geq\dotsm\geq\la_i\geq2$. From e.g.\ the basic identity
  $$ \sum_{m\geq 0}h_mt^m = \exp \sum_{i\geq 1}p_i\frac{t^i}{i} $$
it follows that $\frac{\partial h_m}{\partial p_1} = h_{m-1}$.
Then (since $h_1=p_1$)
\begin{align*}\label{F:action}
\frac{\partial}{\partial p_1}p_1\cdot h_\la&=\frac{\partial}{\partial
p_1}h_{\la_1}\dotsm h_{\la_k}\dotsm h_{\la_i}
              h_1^{j+1}\\\nonumber
              &=(j+1)h_\la+\sum_{k=1}^i h_{\la_1}\dotsm
              h_{\la_k-1}\dotsm h_{\la_i}h_1^{j+1}.
\end{align*}
We see that $a_{\la\la}=m_1(\la)+1$. Let $\mu$ be the partition such
that \begin{align*}
h_{\la_1}\dotsm h_{\la_k-1}\dotsm h_{\la_i}h_1^{j+1}=h_\mu,\ \ 1\leq
k\leq i .
\end{align*}
In the case that $k=1$ and $\la$ is not an initial element, i.e.,
$\la_1>\la_2$, then $\mu=(\la_1-1,\la_2,\dots, \la_i,1^{j+1})$ and
thus $\la\nwarrow\mu$.

In the other cases, $\mu$ is of the form
$\mu=(\la_1,\la_2,\dots,\la_r-1,\dots, \la_i,1^{j+1})$ for some
$1< r\leq i$. Hence
$T(\mu)=(\la_1+j+1,\la_2,\dots,\la_r-1,\dots, \la_i)$ or
$T(\mu)=(\la_1+j+2,\la_2,\dots, \la_{i-1})$ (when $i=r$ and
$\la_r=2$). Note that $T(\la)=(\la_1+j,\la_2,\dots,\la_i)$. We see
that $T(\mu)>T(\la)$ (in dominance order). This proves (1) and
(2). 
\end{proof}

In the following we set $a'_i=i$ for $i=1,2,\dots,n-1$ and
$a'_n=n+1$.
We separate rows and columns corresponding to different strings and
write $A$ in the block matrix form $A=(A_{kl})_{t\times t}$. We have
the following properties of these $A_{kl}$'s by
Lemma \ref{L:PropertyOfA}. 
 \be
 \item[(1)] For $k>l$, $A_{kl}=0$.
 \item[(2)] $A_{kk}$ is an $i_k\times i_k$ lower triangular matrix.
Its diagonal entries are $a'_1,a'_2,\dots, a'_{i_k}$; the entries on
the line right below and parallel to the diagonal are all $1$ and all
the other entries are 0. \ee 
It looks as follows:

\footnotesize
\[ A=\left[\begin{array}{ccccccccccccc}
 a'_1 &   &    &   &b^{12}_{11}   &b^{12}_{12}   &\cdots   & b^{12}_{1i_2}  &  \cdots & b^{1t}_{11}  & b^{1t}_{12} &\cdots&b^{1t}_{1i_t}\\[.1in]
 1  &a'_2 &    &   &b^{12}_{21}&b^{12}_{22}   &\cdots   & b^{12}_{2i_2}
 & \cdots  & b^{1t}_{21}   &b^{1t}_{22}&\cdots&b^{1t}_{2i_t}  \\
    &\ddots  &\ddots  &  &  \cdots & &  &   &   &  \cdots\ &  && \\
    &   &  1 &a'_{i_1} &  b^{12}_{i_11}&b^{12}_{i_12}&\cdots  &
 b^{12}_{i_1i_2}&  \cdots&   b^{1t}_{i_11}&  b^{1t}_{i_12}   &\cdots&
 b^{1t}_{i_1i_t}\\[.1in]
    &   &    &   &a'_1 &   &   &   & \cdots  & b^{2t}_{11}  & b^{2t}_{12} &\cdots&b^{2t}_{1i_t}\\[.1in]
    &   &    &   &1  &a'_2 &   &   &   \cdots  & b^{2t}_{21}   &b^{2t}_{22}&\cdots&b^{2t}_{2i_t} \\
    &   &    &   &   &\ddots &\ddots &   &  & \cdots& &&  \\
    &   &    &   &   &   &1  &a'_{i_2} &  \cdots  &   b^{2t}_{i_21}&
 b^{2t}_{i_22}   &\cdots& b^{2t}_{i_2i_t} \\
    &   &    &   &   &   &   &  &\ddots &    &&& \\
    &   &    &   &   &   &   &   &   &a'_1 & &&  \\
    &   &    &   &   &   &   &   &   &  1 &a'_2&& \\
        &   &    &   &   &   &   &   &   &   &\ddots&\ddots&\\
            &   &    &   &   &   &   &   &   &   &&1&a'_{i_t}\\

\end{array}\right] .\]\normalsize
Furthermore, we have
\begin{align}\label{strictupper}
b^{kl}_{ij}=0 \mbox{ if }i\leq j,
\end{align}
 i.e., $A_{kl}$ is a strict lower triangular matrix if $k<l$.  The reason is that $b^{kl}_{ij}=a_{\la^{ki}\la^{lj}}$,
  and if it is nonzero, then we have the following by Lemma \ref{L:PropertyOfA}:
 $$i-1=m_1(\la^{ki})\geq m_1(\la^{lj})+1=j-1+1=j.$$
(Here we use that $m_1(\la^{pr})=r-1$ for $\la^{pr}\neq(1^n)$. It is possible that $\la^{ki}=(1^n)$, but then again $i=n\geq i_l+1\geq j+1$.)

Now we replace the $a'_i$ on the diagonal of $A$ with an arbitrary
$f_i=f_i(x)\in\zz[x]$ and change $A$ into a matrix $A(x)$ with entries
in $\zz[x]$. We will apply some row/column operations to $A(x)$ and
transform it into an SNF in $\zz[x]$. (Some of these operations
depend on the $f_i$'s.) This is the same as to say that there are
$P_1(x),Q_1(x)\in \mathrm{SL}(p(n),\zz[x])$  such that
$P_1(x)A(x)Q_1(x)$ is an SNF in $\zz[x]$.

 Notice that the original matrix $M$ is equal to $PAP^{-1}$ for some
 $P\in \mathrm{SL}(p(n),\zz)$. If we take $f_i=a'_i+x$ (which is $a_i(x)$ in our theorem) in the beginning,
 then $A(x)=A+xI_{p(n)}$, and thus the SNF $P_1(x)A(x)Q_1(x)$ is equal
 to
 $P_1(x)(P^{-1}MP+xI_{p(n)})Q_1(x)=P_1(x)P^{-1}(M+xI_{p(n)})PQ_1(x)$,
 as desired.
\subsection{Transformation into an upper triangular matrix}

If we use horizontal lines to separate rows of $A(x)$ corresponding to
different full strings,
then $A(x)$ is partitioned into $t$ submatrices. We see that we
should consider a matrix of the following form:
\[ B=\left[\begin{array}{cccccccc}
 f_1&       &       &      &b_{11}   &b_{12}&\cdots & b_{1m}  \\
 1  &f_2    &       &      &b_{21}   &b_{22}&\cdots & b_{2m}  \\
    &\ddots &\ddots &      &\cdots   &      &       &      \\
    &       &  1    &f_{s} &b_{s1}   &b_{s2}&\cdots & b_{sm}
\end{array}\right].\]
We can apply row operations to $B$ and transform it first into
\[ B_1=\left[\begin{array}{cccccccc}
 1  &f_2     &        &      &  b_{21} &b_{22}&\cdots & b_{2m}\\
    &\ddots  &\ddots  &      &  \cdots &      &       &       \\
    &        &  1     &f_{s} &  b_{s1} &b_{s2}&\cdots & b_{sm} \\
 f_1&        &        &      &b_{11}   &b_{12}&\cdots &  b_{1m}
\end{array}\right],\]
and then into
\[ B_2=\left[\begin{array}{cccccccc}
 1  &f_2    &        &      &  b_{21} &b_{22} &\cdots & b_{2m}  \\
    &\ddots &\ddots  &      &  \cdots &       &       &       \\
    &       &  1     &f_{s} &  b_{s1} &b_{s2} &\cdots & b_{sm} \\
 0  & \cdots &  0      &\alpha&\beta_1  &\beta_2&\cdots & \beta_m
\end{array}\right] ,\]

\noindent with $\alpha=f_1\dotsm f_s$ and
 $$ (-1)^{s-1}\beta_j=b_{1j}-f_1b_{2j}+f_1f_2b_{3j}+\dotsm+(-1)^{s-1}f_1\dotsm
    f_{s-1}b_{sj}. $$
Apply this process to the $t$ submatrices of $A(x)$. We turn $A(x)$
    into the matrix
\[ A_1(x)=
\left[\begin{array}{ccccccccccccc}
 1 & f_2            &         &                &b^{12}_{21}  &b^{12}_{22}   &\cdots
   &b^{12}_{2i_2}   & \cdots  & b^{1t}_{21}    &b^{1t}_{22}  &\cdots        &b^{1t}_{2i_t}\\
   & \ddots         &\ddots   &                &  \cdots     &              &
   &                &         &  \cdots        &             &              & \\
   &                &1        & f_{i_1}        &b^{12}_{i_11}&b^{12}_{i_12} &\cdots
   &b^{12}_{i_1i_2} & \cdots  & b^{1t}_{i_11}  &b^{1t}_{i_12}&\cdots        &b^{1t}_{i_1i_t} \\[.5em]
   &                &         &\al_{i_1}       &\beta^{12}_1 &\beta^{12}_2  &\cdots
   &\beta^{12}_{i_2}&  \cdots &   \beta^{1t}_1 &\beta^{1t}_2 &\cdots        & \beta^{1t}_{i_t}\\[.5em]
   &                &         &                &1            & f_2          &
   &                & \cdots  & b^{2t}_{21}    &b^{2t}_{22}  &\cdots        &b^{2t}_{2i_t}\\
   &                &         &                &             &\ddots        &\ddots
   &                &         & \cdots         &             &              &  \\
   &                &         &                &             &              &1
   & f_{i_2}        & \cdots  & b^{2t}_{i_21}  &b^{2t}_{i_22}&\cdots        &b^{2t}_{i_2i_t} \\[.5em]
   &                &         &                &             &              &
   &\alpha_{i_2}    & \cdots  &  \beta^{2t}_{1}&\beta^{2t}_2 &\cdots        & \beta^{2t}_{i_t} \\
   &                &         &                &             &              &
   &                &\ddots   &                &             &              & \\
   &                &         &                &             &              &
   &                &         &1               &f_2          &              &  \\
   &                &         &                &             &              &
   &                &         &                &\ddots       &\ddots        &\\
   &                &         &                &             &              &
   &                &         &                &             &1             &f_{i_t} \\
   &                &         &                &             &              &
   &                &         &                &             &              &\alpha_{i_t}\\

\end{array}\right] ,\]\normalsize

where $\al_k=f_1f_2\dotsm f_{k}$ and
\begin{align*}
(-1)^{i_k-1}\beta^{kl}_j&=b^{kl}_{1j}-f_1b^{kl}_{2j}+f_1f_2b^{kl}_{3j}+\dotsm+(-1)^{i_k-1}f_1f_2\dotsm
        f_{i_k-1}b^{kl}_{i_kj}.\\
\end{align*}

Recalling that $b^{kl}_{ij}=0$ for $i\leq j$ (see (\ref{strictupper})), we find that
\begin{equation}\label{P:crucial}
f_1f_2\cdots f_j\mid \beta^{kl}_j, \text{ and as a special case,
}\alpha_{i_l}\mid \beta^{kl}_{i_l}.
\end{equation}
This property is crucial for later cancellation.
\begin{remark}\label{R:idea2}
Next we will show that we can cancel the nondiagonal entries without altering the diagonal.
Then by the definition of $\al_k$, we see that the matrix $A(x)$ has the Smith normal form
\begin{align*}
\mathrm{diag}(1,1,\dotsc,1,\al_{j_1},\al_{j_2},\dotsc,\al_{j_t}),
\end{align*}
where $(j_1,\dotsc,j_t)$ is the rearrangement of $i_1,\dotsc, i_t$ in
weakly increasing order. Combining Lemma \ref{L:stringcardinality}, everything in our theorem is clear now.

\end{remark}
\subsection{The cancellation of the nondiagonal entries}
Now we want to cancel the nondiagonal elements, completing the proof.
For those nonzero elements above the diagonal, we can do the
following.
 \be
 \item[(C1)] First apply column operations to cancel the entries on
 the rows with diagonal elements equal to $1$ (starting from the first
 row).
 \item[(C2)] Then apply row operations to cancel the entries on the
 columns with diagonal elements equal to $1$.
 \ee
The matrix turns into the following:

\[ A_2(x)=\left[\begin{array}{ccccccccccccc}
 1 &   &    &   &0&0   &\cdots   & 0 & \cdots  & 0  &0&\cdots&0\\
    & \ddots &\ddots  &  &  \cdots & &  &   &   &  \cdots\ &  && \\
    &&1 &       &0&0   &\cdots   & 0 & \cdots  & 0   &0&\cdots&0  \\
    &   &  &\al_{i_1} &  0&0&\cdots  & \beta^{12}&  \cdots& 0  &  0   &\cdots& \beta^{1t}\\
    &   &    &  &1 &   &   &   &  & 0   &0&\cdots&0\\
    &   &    &   &   &\ddots &\ddots &   &  & \cdots& &&  \\
    &   &    &&   &  &1 &   &  &      0   &0&\cdots&0 \\
    &   &    &   &   &   & &\alpha_{i_2} &   &   0&  0  &\cdots& \beta^{2t} \\
    &   &    &   &   &   &   &  &\ddots &    &&& \\
    &   &    &   &   &   &   &   &   &1 & &&  \\
    &   &    &   &   &   &   &   &   &   &1&& \\
        &   &    &   &   &   &   &   &   &   &&\ddots&\\
            &   &    &   &   &   &   &   &   &   &&&\alpha_{i_t}\\

\end{array}\right]. \]
The only entries we cannot cancel in (C1) and (C2) are those on the
intersection of rows and columns with $\alpha_{i_l}$'s, i.e., the
$\beta^{kl}$'s in $A_2(x)$. If we can prove that each $\beta^{kl}$ is a
multiple of $\al_{i_l}$, then we can apply row operations to cancel
all those $\beta^{kl}$'s, and we are done.

To see this, let us first go back to $A_1(x)$. We know that at the
beginning $\beta^{kl}_{i_l}$ was a multiple of $\al_{i_l}$ by
(\ref{P:crucial}).  Then this entry was changed to $\beta^{kl}$ after
we applied (C1) and (C2). More precisely, (C1) changed it but (C2) did
not.  If we look more closely, we find that terms were added to
$\beta^{kl}_{i_l}$ only when we were doing the column operations to
cancel the nonzero entries below this entry $\beta^{kl}_{i_l}$. We claim
that each term which was added to this entry was actually a multiple
of $\al_{i_l}$. Thus the new term $\beta^{kl}$ in $A_2(x)$ is still a
multiple of $\al_{i_l}$.

Now let us prove our claim. For simplicity we consider only the entry
$\beta^{1t}_{i_t}$. The general
case can be treated similarly. The entries which were below this
$\beta^{1t}_{i_t}$ and which were canceled in (C1) were
$b^{kt}_{ci_t}$ ($2\leq k\leq t-1$, $2\leq c\leq i_k$) together with
$f_{i_t}$ which was right above $\al_{i_t}$.  \be \item[(a)] If
$b^{kt}_{ci_t}$ was nonzero, then $c\geq i_t+1$ by (\ref{strictupper}).
To cancel this $b^{kt}_{ci_t}$, we added $-b^{kt}_{ci_t}$ times the
$\la^{k,(c-1)}$ column (i.e., the column indexed by $\la^{k,(c-1)}$)
to the $\la^{ti_t}$ column. Thus  
$-b^{kt}_{ci_t}\beta^{1k}_{c-1}$ was added to $\beta^{1t}_{i_t}$. By
the fact (\ref{P:crucial}), $f_1\dotsm f_{c-1}\mid
\beta^{1k}_{c-1}$. But $c-1\geq i_t$, so this term added to $\beta^{1t}_{i_t}$
did have $f_1\dotsm f_{i_t}=\al_{i_t}$ as a factor.
 \item[(b)] To cancel
$f_{i_t}$, we added $-f_{i_t}$ times the $\la^{t,i_t-1}$ column to the
$\la^{ti_t}$ column. This added $-f_{i_t}\beta^{1t}_{i_t-1}$ to
$\beta^{1t}_{i_t}$. But again $f_1\dotsm f_{i_t-1}\mid \beta^{1t}_{i_t-1}$ by
(\ref{P:crucial}); we see that this term added is a multiple of
$\al_{i_t}$. \qed
 
\ee


\section{A conjecture.}
We conjecture that our theorem can be generalized to the action
$k\frac{\partial}{\partial p_k}p_k$ for $k\geq 1$.
\begin{conjecture}\label{conj}
Let $M_k^{(n)}$ be the matrix of the map $k\frac{\partial}{\partial
p_k}p_k$ with respect to an integral basis for homogeneous symmetric
functions of degree $n$. Then there exists
$P(x),Q(x)\in\mathrm{SL}(p(n),\zz[x])$ such that
$P(x)(M_k^{(n)}+xI_{p(n)})Q(x)$ is the diagonal matrix
$\mathrm{diag}(f_1(x),\dotsc,f_{p(n)}(x))$, where $f_i(x)$ may be
described as follows. Let $\mathcal{M}$ be the multiset
of all numbers $m_k(\lambda)$ for $\lambda\vdash n$. First,
$f_{p(n)}(x)$ is a product of factors $x+k(a_i+1)$ where the $a_i$'s are
the \emph{distinct} elements of $\mathcal{M}$. Then $f_{p(n)-1}(x)$ is
a product of factors $x+k(b_i+1)$ where the $b_i$'s are the
remaining \emph{distinct} elements of $\mathcal{M}$, etc. (After a
while we will have exhausted all the elements of $\mathcal{M}$. The
remaining diagonal elements are the empty product 1.)
\end{conjecture}


We can prove the following special case of the above conjecture.
The proof is based on the result that for a partition $\lambda\vdash
n$ there is at most one $k$-border strip if and only if $k>n/2$,
though we omit the details here.

 \begin{proposition}
If $k>n/2$, then an SNF of $M_k^{(n)}+xI_{p(n)}$ over $\zz[x]$ is
given by
\begin{equation*}\label{F:snf}
\mathrm{diag}(1,\dotsc,1,x+k,\dotsc,x+k,(x+k)(x+2k),\dotsc,(x+k)(x+2k)),
\end{equation*}
where there are $p(n-k)$ $1$'s and $p(n-k)$ $(x+k)(x+2k)$'s.
 \end{proposition}

Thus it is known that Conjecture~\ref{conj} is true for $k=1$ or $k>n/2$.

\bibliographystyle{amsalpha}

\end{document}